%% file: ramsey.tex
\documentclass{amsart}

\usepackage{graphicx,tikz,amssymb,mathtools,tabularx,algorithm,enumerate,enumitem}
\usetikzlibrary{arrows,calc}
\usepackage[justification=centering]{caption}
\usepackage[noend]{algpseudocode}

\makeatletter
\def\BState{\State\hskip-\ALG@thistlm}
\makeatother

\newtheorem{theorem}{Theorem}[section]
\newtheorem{lemma}[theorem]{Lemma}

\newtheorem{corr}[theorem]{Corollary}

\author{Sam Beilis}
\email{beiliss@kean.edu}

\author{Israel R. Curbelo}
\email{israel.curbelo@kean.edu}

\address{Department of Mathematical Sciences, Kean University, Union, NJ 07083}

\title[On the asymptotic behavior of online Ramsey numbers]{On the asymptotic behavior of online Ramsey numbers for stars, paths and cycles}

\begin{document}

\begin{abstract}
The online Ramsey game for graphs $G$ and $H$ is played on the infinite complete graph $K_\mathbb{N}$. Each round, Builder chooses an edge, and Painter colors it red or blue. The online Ramsey number $\tilde{r}(G,H)$ is the smallest integer $t$ for which Builder has a strategy that guarantees a red copy of $G$ or a blue copy of $H$ in at most $t$ rounds. We show that for every fixed $k$, there are constants $\lambda_1$ and $\lambda_2$ such that $\tilde{r}(P_k,P_n)/n$ and $\tilde{r}(P_k,C_n)/n$ converge to $\lambda_1$, and $\tilde{r}(K_{1,k},P_n)/n$ and $\tilde{r}(K_{1,k},C_n)/n$ converge to $\lambda_2$.
\end{abstract}

\maketitle

\input{introduction}
\input{sub}
\input{paths}

\input{stars}

\input{main}

\input{remarks}

\bibliographystyle{acm}
\bibliography{ramsey}

\end{document}

%% file: introduction.tex
\section{Introduction}

The online Ramsey problem for graphs $G$ and $H$ is defined as a two-player game between Builder and Painter. The game is played in rounds on the infinite complete graph $K_\mathbb{N}$. Each round, Builder chooses an edge and Painter immediately and irrevocably colors the edge red or blue. The \emph{online Ramsey number} $\tilde{r}(G,H)$ is the smallest integer $t$ for which Builder has a strategy that guarantees a red copy of $G$ or a blue copy of $H$ in at most $t$ rounds, regardless of the choices that Painter makes.
The online Ramsey number $\tilde{r}(G,H)$ is always bounded above by the \emph{size Ramsey number} $\hat{r}(G,H)$ introduced by Erd\H{o}s, Faudree, Rousseau, and Schelp \cite{efrs-78} which is the smallest integer $t$ for which there exists a graph $G$ with $t$ edges such that every 2-edge-coloring of $G$ with colors red and blue results in a red copy of $G$ or a blue copy of $H$. However, this bound is often far from optimal. 

Significant attention has been given to the online Ramsey problem involving stars, paths and cycles. Beck \cite{bec-83} showed that the size Ramsey number $\hat{r}(P_n,P_n)$ is linear which implies that the online Ramsey number $\tilde{r}(P_n,P_n)$ is linear. Grytczuk, Kierstead and Pra\l{}at \cite{gry-kie-pra-08} and Pra\l{}at \cite{pra-08,pra-12} computed exact values for $\tilde{r}(P_k,P_n)$ when $\max\{k,n\}\leq 9$. They also showed that for any positive integers $k$ and $n$, $n+k-3\leq\tilde{r}(P_k,P_n)\leq2n+2k-7$. These remain the currently best known bounds when $k=n$. Cyman, Dzido, Lapinskas, and Lo \cite{cdll-15} showed that $\tilde{r}(P_3,P_n)=\left\lceil\frac{5(n-1)}{4}\right\rceil$ and $\tilde{r}(P_4,P_n)\leq\left\lceil\frac{7n}{5}-1\right\rceil$ for any integer $n$. In \cite{cdll-15}, they conjectured that their bound for $\tilde{r}(P_4,P_n)$ was tight. This was later independently confirmed by both Bednarska-Bzd\c{}ega \cite{bed-24} and Zhang and Zhang \cite{zha-zha-23}. In \cite{cdll-15}, Cyman, Dzido, Lapinskas, and Lo also showed that $\tilde{r}(P_k,P_n)\geq \frac{3}{2}n+\frac{k}{2}-\frac{7}{2}$ for any integers $k$ and $n$ when $k\geq5$, and they also conjectured that this bound is tight. In \cite{bed-24}, Bednarska-Bzd\c{}ega showed that $\tilde{r}(P_k,P_n)\leq \frac{5}{3}n+12k$, and in \cite{mon-por-24}, Mond and Portier disprove the conjecture of Cyman, Dzido, Lapinskas, and Lo by showing that $\tilde{r}(P_k,P_n)\geq \frac{5}{3}n+\frac{k}{9}-4$ for $k\geq 10$. This lower bound matches the upper bound in \cite{bed-24} up to an additive constant, determining the asymptotic value of the online Ramsey number $\tilde{r}(P_k,P_n)$. In particular, they showed that $\lim_{n\to\infty}\frac{\tilde{r}(P_k,P_n)}{n}=\frac{5}{3}$ for $k\geq10$. In \cite{mon-por-24}, Mond and Portier state that it is unknown whether the asymptotic value of the online Ramsey number $\tilde{r}(P_k,P_n)$ even exists for $5\leq k\leq 9$. 

We show that for every positive integer $k$, the asymptotic value of the online Ramsey number $\tilde{r}(P_k,P_n)$ exists.
\begin{theorem}\label{thm1}
    For every positive integer $k$, $\lim_{n\to\infty}\frac{\tilde{r}(P_k,P_n)}{n}$ exists.
\end{theorem}

Bla\v{z}ej, Dvo\v{r}\'{a}k and Valla \cite{bla-dvo-val-19} showed that $\tilde{r}(C_k,C_n)$ and $\tilde{r}(C_n,C_n)$ are linear in $n$. The best known bounds on $\tilde{r}(C_k,C_n)$ are 
\begin{equation*}\label{eq1}
    2n-1\le\tilde{r}(C_k, C_n) \le 2n + 20k \text{ for even $k \ge 4$ and $n \ge 3k$,}
\end{equation*}
and
\begin{equation*}
2.6n - 3 <\tilde{r}(C_k, C_n) \le 3n + \log_2 n + 50k \text{ for odd $k \ge 3$ and $n \ge 8k.$}
\end{equation*}
Cyman, Dzido, Lapinskas, and Lo \cite{cdll-15} showed that for every connected graph $H$, Painter can avoid creating a monochromatic cycle or a monochromatic copy of $H$ for at least $|V(H)| + |E(H)| - 1$ rounds, providing the lower bound for the even case above. The lower bound in the odd case was proved by Adamski and Bednarska-Bzd\c{e}ga \cite{ada-bed-24}, while both upper bounds were proved by Adamski, Bednarska-Bzd\c{e}ga, and Bla\v{z}ej \cite{ada-bed-bla-24}. Although the exact value of $\tilde{r}(C_k,C_n)$ is unknown for even $k$, the asymptotic value has been shown to be 2. In particular, if $k\geq 4$ is even, then
\[
    \lim_{n\to\infty} \frac{\tilde{r}(C_k,C_n)}{n}=2.
\]
The fact that $\tilde{r}(C_k,P_n)\leq\tilde{r}(C_k,C_n)$, along with the result in \cite{cdll-15}, implies that the asymptotic value for $\tilde{r}(C_k,P_n)$ is also $2$ when $k$ is even. However, the case when $k$ is odd is not as simple, and the asymptotic value remains unknown or whether it even exists.

On the other hand, not much is known in regards to $\tilde{r}(P_k,C_n)$. Cyman, Dzido, Lapinskas, and Lo \cite{cdll-15} showed that $\tilde{r}(P_3,C_n)=\left\lceil\frac{5n}{4}\right\rceil$ for $n\geq 3$. The online Ramsey number $\tilde{r}(P_k,P_n)$ provides a lower bound for each $k$. We show that the asymptotic value of $\tilde{r}(P_k,C_n)$ exists, and that it is equal to the asymptotic value of $\tilde{r}(P_k,P_n)$.
\begin{theorem}\label{thm2}
    For every positive integer $k$, $\lim_{n\to\infty}\frac{\tilde{r}(P_k,C_n)}{n}$ exists.
\end{theorem}
\begin{theorem}\label{thm3}
    For every positive integer $k$, there is an $L>0$ such that \[\lim_{n\to\infty}\frac{\tilde{r}(P_k,P_n)}{n}=L=\lim_{n\to\infty}\frac{\tilde{r}(P_k,C_n)}{n}.\]
\end{theorem}

The asymptotic values of $\tilde{r}(K_{1,k},P_n)$ and $\tilde{r}(K_{1,k},C_n)$ have been shown to be between $\frac{k}{4}$ and $k$ for any $k\ge 1$ by Grytczuk, Kierstead and Pra\l{}at \cite{gry-kie-pra-08}. Latip and Tan \cite{lat-tan-21} showed that the asymptotic value of $\tilde{r}(K_{1,3},P_n)$ is between $\frac{3}{2}$ and $\frac{5}{3}$. They conjectured that $\frac{3}{2}$ lower bound is tight. Song, Wang and Zhang \cite{son-wan-zha-25} recently confirmed this, and Zhi and Zhang \cite{zhi-zha-26} showed that the same holds true for $\tilde{r}(K_{1,k},C_n)$. We show that the asymptotic values of $\tilde{r}(K_{1,k},C_n)$ and $\tilde{r}(K_{1,k},P_n)$ exist for all $k\ge 1$ and that the values are equal.

\begin{theorem}\label{thm4}
    For every positive integer $k$, $\lim_{n\to\infty}\frac{\tilde{r}(K_{1,k},P_n)}{n}$ exists.
\end{theorem}
\begin{theorem}\label{thm5}
    For every positive integer $k$, $\lim_{n\to\infty}\frac{\tilde{r}(K_{1,k},C_n)}{n}$ exists.
\end{theorem}
\begin{theorem}\label{thm6}
    For every positive integer $k$, there is an $L>0$ such that \[\lim_{n\to\infty}\frac{\tilde{r}(K_{1,k},P_n)}{n}=L=\lim_{n\to\infty}\frac{\tilde{r}(K_{1,k},C_n)}{n}.\]
\end{theorem}

Our paper is organized as follows. In Section \ref{sec:sub}, we define a weaker type of subadditivity called almost subadditivity and present a stronger version of Fekete's Lemma which only requires a sequence to be eventually almost subadditive in order to prove the existence of the asymptotic values in this paper. In Section \ref{sec:pre}, we prove a set of preliminary results in the form of inequalities by constructing strategies for Builder. Finally, in Section \ref{sec:main}, we combine these inequalities with the stronger version of Fekete's Lemma to prove our main results. 

%% file: sub.tex
\section{Subadditivity}\label{sec:sub}

A sequence $\{a_n\}_{n=1}^\infty$ is said to be \emph{subadditive} if $a_{m+n}\leq a_m+a_n$ for any pair of positive integers $m,n$. We say that a sequence $\{a_n\}_{n=1}^\infty$ is \emph{almost subadditive} if there is a constant $C$ such that $a_{m+n}\leq a_m+a_n+C$ for any pair of positive integers $m,n$ satisfying $0.5n\leq m\leq 2n$. We provide the original version of Fekete's Lemma below.

\begin{lemma}[Fekete's Lemma \cite{fek-23}] 
    If $\{a_n\}_{n=1}^\infty$ is subadditive, then $\lim_{n\to\infty}\frac{a_n}{n} = \inf_{n\ge 1} \frac{a_n}{n}$.
\end{lemma}

Fekete's Lemma implies that if we can show that $\tilde{r}(P_k,P_n)$ and $\tilde{r}(P_k,C_n)$ are subadditive, then $\tilde{r}(P_k,P_n)/n$ and $\tilde{r}(P_k,C_n)/n$ converge, respectively. Unfortunately, neither $\tilde{r}(P_k,P_n)$ nor $\tilde{r}(P_k,C_n)$ are subadditive. However, de Bruijn and Erd{\H{o}}s \cite{bru-erd-52} weaken the condition of subadditivity.

\begin{lemma}[de Bruijn and Erd{\H{o}}s \cite{bru-erd-52}]\label{lembe}
    If $\{a_n\}_{n=1}^\infty$ is almost subadditive, then $\lim_{n\to\infty}\frac{a_n}{n} = \inf_{n\ge 1} \frac{a_n+C}{n}$.
\end{lemma}

We say that a sequence $\{a_n\}_{n=1}^\infty$ is \emph{eventually almost subadditive} if there are constants $C$ and $N$ such that $a_{m+n}\leq a_m+a_n+C$ for any pair of positive integers $m,n$ satisfying $0.5n\leq m\leq 2n$ and $\min\{m,n\}>N$. It is easy to see that if a sequence is eventually almost subadditive, then it is almost subadditive, and since $\tilde{r}(G,H)>0$ always, we use the following.

\begin{corr}\label{lemfek}
    If $\{a_n\}_{n=1}^\infty$ is eventually almost subadditive and $a_n>0$ for all $n$, then $\lim_{n\to\infty}\frac{a_n}{n}$ exists.
\end{corr}

In this paper, we show that $\tilde{r}(P_k,P_n)$, $\tilde{r}(P_k,C_n)$, $\tilde{r}(K_{1,k},P_n)$ and $\tilde{r}(K_{1,k},C_n)$ are eventually 
almost subadditive. 

%% file: paths.tex
\section{Preliminary Results}\label{sec:pre}

In this section we prove a set of inequalities by constructing strategies for Builder. These strategies combine to form bigger strategies that prove each of our main results. In our version of the online Ramsey problem, Builder is allowed to choose a previously chosen edge. Note that this does not make Builder any stronger, however, it makes many of the proofs much simpler. We also make little attempt to minimize the additive constants as they make no difference in our main results and would only complicate the arguments. 
Our first inequality is trivial, however, we present it for completeness.

\begin{lemma}\label{lem1}
    Let $H$ be a fixed graph. For any positive integer $n$, 
    \[\tilde{r}(H,P_{n}) \leq\tilde{r}(H,C_{n}).\]
\end{lemma}
\begin{proof}
    The proof is trivial since for any fixed graph $H$ and any positive integer $n$, Builder can force either a red $H$ or a blue $C_n$ in at most $\tilde{r}(H,C_{n})$ rounds, and a blue $P_n$ is contained in every blue $C_n$.
\end{proof}

\subsection{Paths vs Paths and Cycles}

We present a strategy for Builder that, starting with a blue $P_n$, either extends it to a blue $P_{n+t}$ or forces a red $P_k$ in at most $2k+2t$ more moves. This strategy is motivated by the strategy in \cite{gry-kie-pra-08} that forces either a red $P_k$ or a blue $P_n$ in at most $2k+2n-7$ rounds. The details are a bit tedious, but Figure \ref{fig:ext} demonstrates the main idea. We refer to a vertex as \emph{isolated} if it is not incident to an edge previously chosen by Builder.

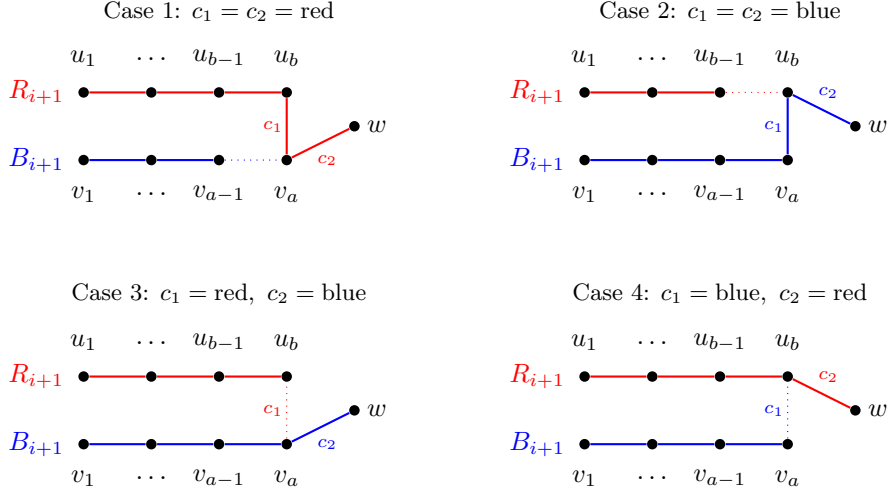
\begin{figure}
\centering
\input{fig1.tikz}
   \caption{A Builder strategy to extend the total length of disjoint red and blue paths by 1, every 2 rounds. Builder joins the paths by choosing edge $v_au_b$, and depending on the color $c_1$ assigned by Painter, Builder either chooses edge $v_aw$ or $u_bw$.} 
   \label{fig:ext}
\end{figure}
    
\begin{lemma}\label{lem2}
    For any positive integers $k$, $n$ and $t$, 
    \[\tilde{r}(P_k,P_{n+t}) \leq\tilde{r}(P_k,P_{n})+2k+2t.\]
\end{lemma}
\begin{proof}
    Let $k$, $n$ and $t$ be positive integers. We show that Builder has a strategy that forces a red $P_k$ or a blue $P_{n+t}$ in at most $\tilde{r}(P_k,P_{n})+2k+2t$ rounds. 

    Builder first plays the strategy that forces a red $P_k$ or a blue $P_n$ in $N\leq\tilde{r}(P_k,P_{n})$ rounds. If Painter creates a red $P_k$, then Builder wins and we are done, so we may assume that Painter creates a blue $P_n$. Let 
    \[
    B_0=v_1v_2\ldots v_{n-1}v_n
    \]
    be the blue $P_n$, and let $u_1$ be an isolated vertex. Notice that $R_0=u_1$ is a red $P_1$.

    Builder chooses the next $2k+2t$ edges as follows. Suppose in round $N+2i$, Builder has successfully forced two disjoint paths 
    \[B_i=v_1\ldots v_a\]
    and
    \[R_i=u_1\ldots u_b\]
    so that $B_i$ is a blue $P_a$ and $R_i$ is a red $P_b$. In round $N+2i+1$, Builder chooses the edge $v_a u_b$. Let $w$ be an isolated vertex. If Painter colors the edge $v_a u_b$ blue, then in round $N+2i+2$, Builder chooses the edge $u_b w$. If Painter colors $v_a u_b$ red, then in round $N+2i+2$, Builder chooses the edge $v_a w$. Let $c_1$ be the color that Painter chooses in round $N+2i+1$, and let $c_2$ be the color that Painter chooses in round $N+2i+2$. There are a total of four cases for $a>1$ and $b>1$.
    
    \begin{enumerate}
        \item If $c_1=c_2=\text{red}$, then $B_{i+1}=v_1\ldots v_{a-1}$ and $R_{i+1}=u_1\ldots u_b v_a w$.
        \item If $c_1=c_2=\text{blue}$, then $B_{i+1}=v_1\ldots v_{a} u_b w$ and $R_{i+1}=u_1\ldots u_{b-1}$.
        \item If $c_1=\text{red}$ and $c_2=\text{blue}$, then $B_{i+1}=v_1\ldots v_{a} w$ and $R_{i+1}=u_1\ldots u_{b}$.
        \item If $c_1=\text{blue}$ and $c_2=\text{red}$, then $B_{i+1}=v_1\ldots v_{a}$ and $R_{i+1}=u_1\ldots u_{b} w$.
    \end{enumerate}
    Let $x$ be an isolated vertex. If $a=1$, $B_{i+1}=x$ for Case 1, and if $b=1$, $R_{i+1}=x$ for Case 2.
    
    By round $N+2k+2t$, Builder has forced two disjoint paths $R_{k+t}$ and $B_{k+t}$ so that $R_{k+t}$ is a red $P_{a}$ and $B_{k+t}$ is a blue $P_{b}$ with $a+b\geq n+k+t+1$. By the pigeonhole principle, $a\geq k$ or $b\geq n+t$. This completes the proof.
\end{proof}

Our next strategy shows that starting with two disjoint long blue paths, in at most $k$ more rounds, Builder can force a red $P_k$ or connect the two blue paths to form a longer blue path, while losing at most $k$ in total length.

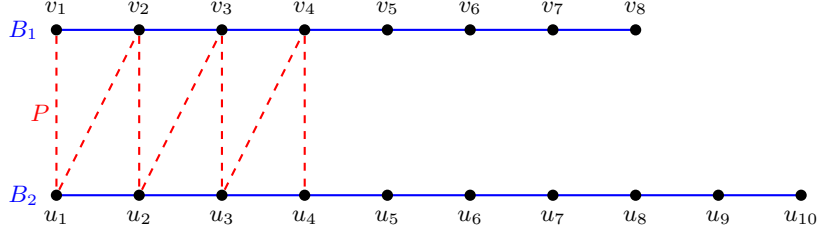
\begin{figure}
        \centering
        \input{fig2.tikz}
        \caption{The Builder strategy in Lemma \ref{lem3} forces Painter to either join the two long blue paths or create a short red path.}
        \label{fig:join}
    \end{figure}
    
\begin{lemma}\label{lem3}
    For any positive integers $k$, $m$ and $n$ satisfying $\min\{m,n\}>k/2$, 
    \[\tilde{r}(P_k,P_{m+n-k}) \leq\tilde{r}(P_k,P_{m})+\tilde{r}(P_k,P_{n})+k.\]  
\end{lemma}
\begin{proof}
    Let $k$, $m$ and $n$ be positive integers. We show that Builder has a strategy that forces a red $P_k$ or a blue $P_{m+n-k}$ in at most $\tilde{r}(P_k,P_{m})+\tilde{r}(P_k,P_{n})+k$ rounds. 

    Builder first plays the strategy that forces a red $P_k$ or a blue $P_m$ in at most $\tilde{r}(P_k,P_{m})$ rounds. If Painter creates a red $P_k$, then Builder wins and we are done, so we may assume that Painter creates a blue $P_m$. Let 
    \[
    B_1=v_1v_2\ldots v_{m-1}v_m
    \]
    be the blue $P_m$.
    
    Builder then plays the strategy that forces a red $P_k$ or a blue $P_n$ in at most $\tilde{r}(P_k,P_{n})$ rounds. If Painter creates a red $P_k$, then Builder wins and we are done, so we may assume that Painter creates a blue $P_n$. Let 
    \[
    B_2=u_1u_2\ldots u_{n-1}u_n
    \]
    be the blue $P_n$.

    Finally, Builder chooses the edges forming the path
    \[
    P=v_1u_1v_2u_2v_3u_3\ldots 
    \]
    until Painter colors an edge blue. If Painter colors the first $k-1$ edges red, then Builder has forced a red $P_k$ and we are done. 
    Hence, we may assume that Painter colors one of the first $k-1$ edges blue. In that case, this edge connects $B_1$ and $B_2$ forming a blue $P_N$ with $N>m+n-k$. Since any blue $P_N$ contains a blue $P_{m+n-k}$, the proof is complete.
    
\end{proof}


Our last strategy for paths consists of two phases. In the first phase, starting with a long blue path, in at most $k$ rounds, Builder can connect the ends of the two blue paths to form a large blue cycle, while losing at most $k$ in length, or force a red $P_k$. In the second phase, starting with a long blue cycle, Builder can force a red $P_k$, or a blue chord within the large blue cycle forming a blue $C_{n-k}$. Note that if it was sufficient to force a large blue cycle, then showing the existence of the asymptotic value for $\tilde{r}(P_k,C_n)$ would be no harder than showing the existence of the asymptotic value for $\tilde{r}(P_k,P_n)$. However, unlike paths, a cycle does not contain any smaller cycles. It is forcing a blue cycle of a specific size that forces us to be much more meticulous with our construction.

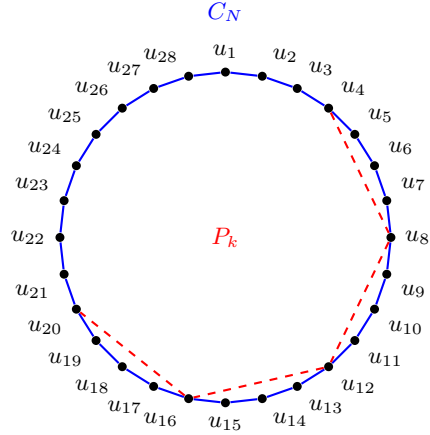
\begin{figure}
        \centering
        \input{fig3.tikz}
        \caption{A Builder strategy forcing Painter to create a blue $C_{n-k}$ or a red $P_k$, starting with a blue $C_N$ where $n-k< N\leq n$. In this example, $n=30$, $k=5$ and $N=28$.}
        \label{figc}
    \end{figure}

\begin{lemma}\label{lem4}
    For any positive integers $k$ and $n$ satisfying $n>k^2$,
    \[\tilde{r}(P_k, C_{n-k})\leq \tilde{r}(P_k, P_n)+2k.\]
\end{lemma}
\begin{proof}
    Let $k$ and $n$ be positive integers. We show that Builder has a strategy that forces a red $P_k$ or a blue $C_{n-k}$ in at most $\tilde{r}(P_k, P_n)+2k$ rounds. 
    
    Builder first plays the strategy that forces a red $P_k$ or a blue $P_n$ in at most $\tilde{r}(P_k,P_{n})$ rounds. If Painter creates a red $P_k$, then Builder wins and we are done, so we may assume that Painter creates a blue $P_n$. Let 
    \[
    B=v_1v_2\ldots v_{n-1}v_n
    \]
    be the blue $P_n$.      

    Next, Builder chooses the edges forming the path
    \[
    P'=v_1v_nv_2v_{n-1}v_3v_{n-2}\ldots 
    \]
    until Painter colors an edge blue. If Painter colors the first $k-1$ edges red, then Builder has forced a red $P_k$, and we are done. 
    So we may assume that Painter colors one of the first $k-1$ edges blue.
    Hence, we may assume that there is an edge in $P'$ that was colored blue. In that case, this edge connects $B$ to itself forming a blue $C_N$ with $n-k< N\leq n$. 
    Let $\alpha=N+k-n+1$ and let
    \[
    C=u_1u_2\ldots u_{N-1}u_Nu_1
    \]
    be the blue $C_N$.
    
    Finally, Builder chooses the edges forming the path
    \[
    P=u_{\alpha}u_{2\alpha}\ldots u_{k\alpha}.
    \]
    If Painter colors every edge in $P$ red, then $P$ is a red $P_k$, and we are done. Hence, we may assume that there is an edge in $P$ that was colored blue. In that case, this edge is a chord in $C$ which forms a blue $C_{n-k}$ (See Figure \ref{figc}.) In particular, if edge $u_{(i-1)\alpha}u_{i\alpha}$ is colored blue for some $i\leq k$, then \[
        u_{(i-1)\alpha} u_{i\alpha} u_{i\alpha+1}\ldots u_{N-1} u_N u_1 u_2 \ldots u_{(i-1)\alpha} 
    \]
    is a blue $C_{n-k}$.
\end{proof}

%% file: fig1.tikz
\begin{tikzpicture}[
    scale=0.9,
    every node/.style={circle, fill, inner sep=1.4pt},
    lab/.style={draw=none, fill=none, rectangle},
    >=stealth
]

\begin{scope}[xshift=0cm, yshift=0cm]
    \node[lab,above] at (2,2.0) {\small Case 1: $c_1=c_2=\mathrm{red}$};

    \foreach \i in {0,1,2,3}{
        \node (u\i) at (\i,1) {};
        \node (v\i) at (\i,0) {};
    }
    \node[lab] at (0,1.5) {$u_1$};
    \node[lab] at (2,1.5) {$u_{b-1}$};
    \node[lab] at (1,1.5) {$\ldots$};
    \node[lab] at (3,1.5) {$u_b$};
    \node[lab] at (0,-0.5) {$v_1$};
    \node[lab] at (2,-0.5) {$v_{a-1}$};
    \node[lab] at (1,-0.5) {$\ldots$};
    \node[lab] at (3,-0.5) {$v_a$};
   
    \draw[red, thick] (u0)--(u1)--(u2)--(u3);

    \draw[blue, thick] (v0)--(v1)--(v2);
    \draw[blue,dotted] (v2)--(v3);

    \node[label=right:$w$] (w) at (4,0.5) {};
    \draw[red, thick] (v3)--(u3);
    \draw[red, thick] (v3)--(w);

    \node[lab,red] at (2.8,0.5) {$\scriptstyle c_1$};
    \node[lab,red] at (3.6,0) {$\scriptstyle c_2$};
    \node[lab,red] at (-0.7,1) {$R_{i+1}$};
    \node[lab,blue] at (-0.7,0) {$B_{i+1}$};

\end{scope}

\begin{scope}[xshift=7.4cm, yshift=0cm]
    \node[lab,above] at (2,2.0) {\small Case 2: $c_1=c_2=\mathrm{blue}$};

    \foreach \i in {0,1,2,3}{
        \node (u\i) at (\i,1) {};
        \node (v\i) at (\i,0) {};
    }
    \node[lab] at (0,1.5) {$u_1$};
    \node[lab] at (2,1.5) {$u_{b-1}$};
    \node[lab] at (1,1.5) {$\ldots$};
    \node[lab] at (3,1.5) {$u_b$};
    \node[lab] at (0,-0.5) {$v_1$};
    \node[lab] at (2,-0.5) {$v_{a-1}$};
    \node[lab] at (1,-0.5) {$\ldots$};
    \node[lab] at (3,-0.5) {$v_a$};
   
    \draw[red, thick] (u0)--(u1)--(u2);
    \draw[red,dotted] (u2)--(u3);

    \draw[blue, thick] (v0)--(v1)--(v2)--(v3);

    \node[label=right:$w$] (w) at (4,0.5) {};
    \draw[blue, thick] (v3)--(u3);
    \draw[blue, thick] (u3)--(w);

    \node[lab,blue] at (2.8,0.5) {$\scriptstyle c_1$};
    \node[lab,blue] at (3.6,1) {$\scriptstyle c_2$};

    \node[lab,red] at (-0.7,1) {$R_{i+1}$};
    \node[lab,blue] at (-0.7,0) {$B_{i+1}$};

\end{scope}

\begin{scope}[xshift=0cm, yshift=-4.2cm]
    \node[lab,above] at (2,2.0) {\small Case 3: $c_1=\mathrm{red},\ c_2=\mathrm{blue}$};

    \foreach \i in {0,1,2,3}{
        \node (u\i) at (\i,1) {};
        \node (v\i) at (\i,0) {};
    }
    \node[lab] at (0,1.5) {$u_1$};
    \node[lab] at (2,1.5) {$u_{b-1}$};
    \node[lab] at (1,1.5) {$\ldots$};
    \node[lab] at (3,1.5) {$u_b$};
    \node[lab] at (0,-0.5) {$v_1$};
    \node[lab] at (2,-0.5) {$v_{a-1}$};
    \node[lab] at (1,-0.5) {$\ldots$};
    \node[lab] at (3,-0.5) {$v_a$};
   
    \draw[red, thick] (u0)--(u1)--(u2)--(u3);

    \draw[blue, thick] (v0)--(v1)--(v2)--(v3);

    \node[label=right:$w$] (w) at (4,0.5) {};
    \draw[red,dotted] (v3)--(u3);
    \draw[blue, thick] (v3)--(w);

    \node[lab,red] at (2.8,0.5) {$\scriptstyle c_1$};
    \node[lab,blue] at (3.6,0) {$\scriptstyle c_2$};

    \node[lab,red] at (-0.7,1) {$R_{i+1}$};
    \node[lab,blue] at (-0.7,0) {$B_{i+1}$};
    
\end{scope}

\begin{scope}[xshift=7.4cm, yshift=-4.2cm]
    \node[lab,above] at (2,2.0) {\small Case 4: $c_1=\mathrm{blue},\ c_2=\mathrm{red}$};

    \foreach \i in {0,1,2,3}{
        \node (u\i) at (\i,1) {};
        \node (v\i) at (\i,0) {};
    }
    \node[lab] at (0,1.5) {$u_1$};
    \node[lab] at (2,1.5) {$u_{b-1}$};
    \node[lab] at (1,1.5) {$\ldots$};
    \node[lab] at (3,1.5) {$u_b$};
    \node[lab] at (0,-0.5) {$v_1$};
    \node[lab] at (2,-0.5) {$v_{a-1}$};
    \node[lab] at (1,-0.5) {$\ldots$};
    \node[lab] at (3,-0.5) {$v_a$};
   
    \draw[red, thick] (u0)--(u1)--(u2)--(u3);

    \draw[blue, thick] (v0)--(v1)--(v2)--(v3);

    \node[label=right:$w$] (w) at (4,0.5) {};
    \draw[blue,dotted] (v3)--(u3);
    \draw[red, thick] (u3)--(w);

    \node[lab,blue] at (2.8,0.5) {$\scriptstyle c_1$};
    \node[lab,red] at (3.6,1) {$\scriptstyle c_2$};

    \node[lab,red] at (-0.7,1) {$R_{i+1}$};
    \node[lab,blue] at (-0.7,0) {$B_{i+1}$};

\end{scope}

\end{tikzpicture}

%% file: fig2.tikz
\begin{tikzpicture}[scale=1.1, every node/.style={font=\small}]
    \def\m{8}
    \def\n{10}
    \def\kprime{4}

    \draw[blue,thick] (1,2) -- (\m,2);
    \foreach \i in {1,...,\m} {
        \node[circle, fill=black, inner sep=1.5pt, label=above:$v_{\i}$] (v\i) at (\i,2) {};
    }

    \draw[blue,thick] (1,0) -- (\n,0);
    \foreach \i in {1,...,\n} {
        \node[circle, fill=black, inner sep=1.5pt, label=below:$u_{\i}$] (u\i) at (\i,0) {};
    }

    \node[blue] at (0.6,2) {$B_1$};
    \node[blue] at (0.6, 0) {$B_2$};

    \draw[red, thick, dashed] (v1) -- (u1) -- (v2) -- (u2) -- (v3) -- (u3) -- (v4) -- (u4);

    \node[red] at (.8,1) {$P$};
\end{tikzpicture}

%% file: fig3.tikz
\begin{tikzpicture}[scale=1, every node/.style={font=\small}]
    \def\N{28}          
    \def\nval{40}        
    \def\alpha{5}       
    \def\kmone{4}       

    \def\r{2.2}

    \foreach \i in {1,...,\N} {
        \pgfmathsetmacro{\angle}{90 - (\i-1)*360/\N}
        \node[circle, fill=black, inner sep=1.2pt,
              label={\angle:$u_{\i}$}] (u\i)
              at ({\r*cos(\angle)}, {\r*sin(\angle)}) {};
    }

    \node[blue] at (0,3) {$C_N$};
    \node[red] at (0,0) {$P_k$};
    
    \foreach \i in {1,...,\N} {
        \pgfmathtruncatemacro{\j}{mod(\i,\N) + 1}
        \draw[blue, thick] (u\i) -- (u\j);}


    %

    \draw[red, thick, dashed]
        (u4) -- (u8) -- (u12) -- (u16) -- (u20);

\end{tikzpicture}

%% file: stars.tex
\subsection{Stars vs Paths and Cycles}

We present a Builder strategy that, starting with a blue $P_n$, either extends the blue $P_n$ to a blue $P_{n+1}$ or forces a red $K_{1,k}$ in at most $k$ more moves. Figure \ref{fig:ext2} illustrates the approach. 

\begin{figure}
        \centering
        \input{fig4.tikz}
        \caption{The Builder strategy in Lemma \ref{lem2.1} forces Painter to either create a red star or extend a blue path by 1.}
        \label{fig:ext2}
    \end{figure}
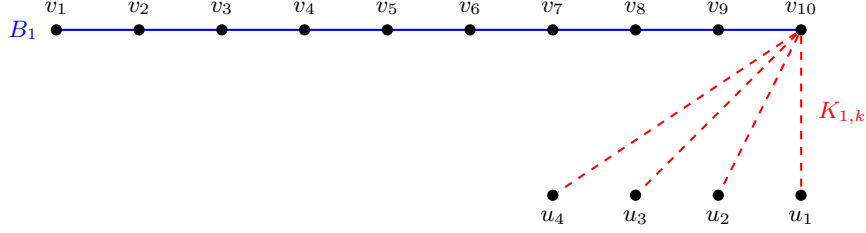
    
\begin{lemma}\label{lem2.1}
    For any positive integers $k$ and $n$, 
    \[\tilde{r}(K_{1,k},P_{n+1}) \leq\tilde{r}(K_{1,k},P_{n})+k.\]
\end{lemma}
\begin{proof}
    Let $k$, and $n$ be positive integers. We show that Builder has a strategy that forces a red $K_{1,k}$ or a blue $P_{n+1}$ in at most $\tilde{r}(K_{1,k},P_{n})+k$ rounds. 

    Builder first plays the strategy that forces a red $K_{1,k}$ or a blue $P_n$ in $N\leq\tilde{r}(K_{1,k},P_{n})$ rounds. If Painter creates a red $K_{1,k}$, then Builder wins and we are done, so we may assume that Painter creates a blue $P_n$. Let 
    \[
    B_0=v_1v_2\ldots v_{n-1}v_n
    \]
    be the blue $P_n$, and let $\{u_1,\ldots,u_k\}$ be a set of isolated vertices. 

    Builder chooses the next $k$ edges as follows. For $i\in\{1,\ldots,k\}$, Builder chooses the edge $v_nu_i$. If Painter colors all $k$ edges red, then Builder forced a red $K_{1,k}$. Hence, we may assume that $v_nu_j$ was colored blue for some $j\in\{1,\ldots,k\}$. Then 
    \[B=v_1v_2\ldots v_nu_j\]
    is a blue $P_{n+1}$.
    
\end{proof}

Repeatedly applying Lemma \ref{lem2.1} provides us with the following.

\begin{lemma}\label{lem2.2}
    For any positive integers $k$, $n$ and $t$, 
    \[\tilde{r}(K_{1,k},P_{n+t}) \leq\tilde{r}(K_{1,k},P_{n})+tk.\]
\end{lemma}

Our next strategy shows that if Builder can force two long blue paths, then in at most $k$ more rounds, Builder can force a red $K_{1,k}$ or connect the two blue paths to form a longer blue path, while losing at most $k-1$ in total size.
    
\begin{lemma}\label{lem2.4}
    For any positive integers $k$, $m$ and $n$ satisfying $m\geq k$, 
    \[\tilde{r}(K_{1,k},P_{m+n-k}) \leq\tilde{r}(K_{1,k},P_{m})+\tilde{r}(K_{1,k},P_{n})+k.\]  
\end{lemma}
\begin{proof}
    Let $k$, $m$ and $n$ be positive integers. We show that Builder has a strategy that forces a red $K_{1,k}$ or a blue $P_{m+n-k}$ in at most $\tilde{r}(K_{1,k},P_{m})+\tilde{r}(K_{1,k},P_{n})+k$ rounds. 

    Builder first plays the strategy that forces a red $K_{1,k}$ or a blue $P_m$ in at most $\tilde{r}(K_{1,k},P_{m})$ rounds. If Painter creates a red $K_{1,k}$, then Builder wins and we are done, so we may assume that Painter creates a blue $P_m$. Let 
    \[
    B_1=v_1v_2\ldots v_{m-1}v_m
    \]
    be the blue $P_m$.
    
    Builder then plays the strategy that forces a red $K_{1,k}$ or a blue $P_n$ in at most $\tilde{r}(K_{1,k},P_{n})$ rounds. If Painter creates a red $K_{1,k}$, then Builder wins and we are done, so we may assume that Painter creates a blue $P_n$. Let 
    \[
    B_2=u_1u_2\ldots u_{n-1}u_n
    \]
    be the blue $P_n$.

    Finally, Builder chooses the edge $u_1v_i$ for each $i\leq k$. If Painter colors all $k$ edges red, then Builder has forced a red $K_{1,k}$ and we are done. 
    So we may assume that Painter colors one of the $k$ edges blue. Suppose that $u_1v_j$ was colors blue for some $j\leq k$. Then
    \[
    B=u_nu_{n-1}\ldots u_2u_1v_jv_{j+1}\ldots v_{m-1}v_m
    \]
    is a blue $P_{m+n-j+1}$. Since $j\leq k$, $B$ contains a blue $P_{m+n-k}$.
    \begin{figure}
        \centering
        \input{fig5.tikz}
        \caption{The Builder strategy in Lemma \ref{lem2.4} forces Painter to either join two long blue paths or create a small red star.}
        \label{fig:join}
    \end{figure}
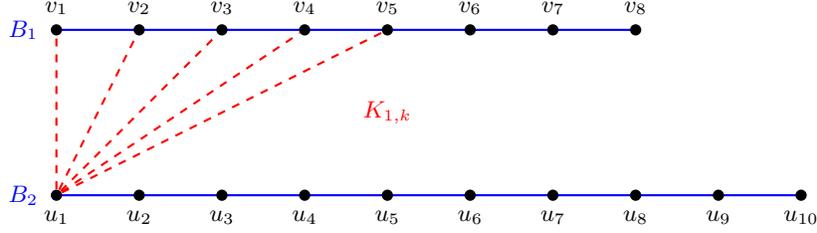
    
\end{proof}


Our last strategy for stars consists of two phases. In the first phase, we show that starting with a long blue path, in at most $k$ more rounds, Builder can connect the ends of the blue path to form a big blue cycle, while losing at most $k-1$ in size, or force a red $K_{1,k}$. In the second phase, starting with one big blue cycle, Builder can force a red $K_{1,k}$ or two blue chords within the big blue cycle forming a blue $C_{n-2k}$. 

\begin{lemma}\label{lem2.5}
    For any positive integers $k$ and $n$ satisfying $n\geq 3k+2$,
    \[\tilde{r}(K_{1,k}, C_{n-2k})\leq \tilde{r}(K_{1,k}, P_n)+3k.\]
\end{lemma}
\begin{proof}
    Let $k$ and $n$ be positive integers. We show that Builder has a strategy that forces a red $K_{1,k}$ or a blue $C_{n-2k}$ in at most $\tilde{r}(K_{1,k}, P_n)+3k$ rounds. 
    
    Builder first plays the strategy that forces a red $K_{1,k}$ or a blue $P_n$ in at most $\tilde{r}(K_{1,k},P_{n})$ rounds. If Painter creates a red $K_{1,k}$, then Builder wins and we are done, so we may assume that Painter creates a blue $P_n$. Let 
    \[
    B=v_1v_2\ldots v_{n-1}v_n
    \]
    be the blue $P_n$.      

    Next, Builder chooses the edge $v_iv_n$ for each $i\in\{1,\ldots,k\}$.
    If Painter colors every edge red, then Builder has forced a red $K_{1,k}$, and we are done. 
    So, we may assume that Painter colors one edge blue.
    In that case, this edge connects $B$ to itself forming a blue $C_N$ with $n-k+1\leq N\leq n$. 
    Let
    \[
    C=u_1u_2\ldots u_{N-1}u_Nu_1
    \]
    be the blue $C_N$.
    
    Finally, Builder chooses the edges $u_1u_{i+2}$ and $u_1u_{n-2k+i}$ for each $i\in\{1,\ldots,k\}$.
    Let $j$ be an integer such that $1\le j\le k$. If Painter colors both $u_1u_{j+2}$ and $u_1u_{n-2k+j}$ blue, then 
    \[
        u_1u_{j+2}u_{j+3}\ldots u_{n-2k+j-1}u_{n-2k+j}u_1
    \]
    is a blue $C_{n-2k}$. Therefore, for each $i\in\{1,\ldots,k\}$, either $u_1u_{i+2}$ or $u_1u_{n-2k+i}$ is colored red by Painter, but this means that Painter created a red $K_{1,k}$. This completes the proof. 

    \begin{figure}
        \centering
        \input{fig6.tikz}
        \caption{The Builder strategy in Lemma \ref{lem2.5} forces Painter to create a blue $C_{n-2k}$ or a red $K_{1,k}$, starting with a blue $C_N$ where $n-k< N\leq n$. In this example, $n=20$, $k=5$ and $N=16$.}
        \label{fig:placeholder}
    \end{figure}
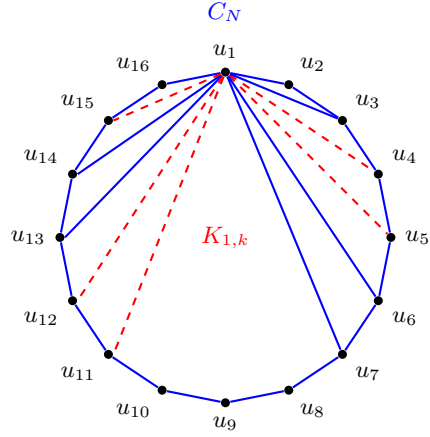
    \label{fig:cycle}
\end{proof}

%% file: fig4.tikz
\begin{tikzpicture}[scale=1.1, every node/.style={font=\small}]
    \def\m{10}
    \def\n{10}
    \def\kprime{4}

    \draw[blue,thick] (1,2) -- (\m,2);
    \foreach \i in {1,...,\m} {
        \node[circle, fill=black, inner sep=1.5pt, label=above:$v_{\i}$] (v\i) at (\i,2) {};
    }
    \foreach \i in {1,...,4} {
    \pgfmathtruncatemacro{\n}{\m-\i+1}
    \node[circle, fill=black, inner sep=1.5pt, label=below:$u_{\i}$] (u\n) at (\n,0) {};
    \draw[red, thick, dashed] (v\m) -- (u\n);
}

    \node[blue] at (0.6,2) {$B_1$};

    \node[red] at (\m+0.5,1) {$K_{1,k}$};
\end{tikzpicture}

%% file: fig5.tikz
\begin{tikzpicture}[scale=1.1, every node/.style={font=\small}]
    \def\m{8}
    \def\n{10}
    \def\kprime{4}

    \draw[blue,thick] (1,2) -- (\m,2);
    \foreach \i in {1,...,\m} {
        \node[circle, fill=black, inner sep=1.5pt, label=above:$v_{\i}$] (v\i) at (\i,2) {};
    }

    \draw[blue,thick] (1,0) -- (\n,0);
    \foreach \i in {1,...,\n} {
        \node[circle, fill=black, inner sep=1.5pt, label=below:$u_{\i}$] (u\i) at (\i,0) {};
    }

    \node[blue] at (0.6,2) {$B_1$};
    \node[blue] at (0.6, 0) {$B_2$};

    \foreach \i in {1,...,5}{
        \draw[red, thick, dashed] (u1) -- (v\i);
    }

    \node[red] at (5,1) {$K_{1,k}$};
\end{tikzpicture}

%% file: fig6.tikz
\begin{tikzpicture}[scale=1, every node/.style={font=\small}]
    \def\N{16}          
    \def\nval{40}        
    \def\alpha{5}       
    \def\kmone{4}       

    \def\r{2.2}

    \foreach \i in {1,...,\N} {
        \pgfmathsetmacro{\angle}{90 - (\i-1)*360/\N}
        \node[circle, fill=black, inner sep=1.2pt,
              label={\angle:$u_{\i}$}] (u\i)
              at ({\r*cos(\angle)}, {\r*sin(\angle)}) {};
    }

    \node[blue] at (0,3) {$C_N$};
    \node[red] at (0,0) {$K_{1,k}$};
    
    \foreach \i in {1,...,\N} {
        \pgfmathtruncatemacro{\j}{mod(\i,\N) + 1}
        \draw[blue, thick] (u\i) -- (u\j);}


    %

    \foreach \i in {4,5}{
    \draw[red, thick, dashed]
        (u1) -- (u\i);
    \pgfmathsetmacro{\butt}{20-\i-2}
    \draw[blue, thick]
        (u1) -- (u\butt);
        }
    \foreach \i in {3,6,7}{
    \pgfmathsetmacro{\butt}{20-\i-2}
    \draw[red, thick, dashed]
        (u1) -- (u\butt);
    \draw[blue, thick]
        (u1) -- (u\i);
        }
\end{tikzpicture}

%% file: main.tex
\section{Main Results}\label{sec:main}

In order to prove Theorem \ref{thm1}, we show that $\tilde{r}(P_k,P_{n})$ is eventually almost subadditive.

\begin{lemma}\label{lem5}
    For any positive integers $k$, $m$ and $n$ satisfying $n> k/2$, 
    \[\tilde{r}(P_k,P_{m+n}) \leq\tilde{r}(P_k,P_{m})+\tilde{r}(P_k,P_{n})+5k.\]
\end{lemma}
\begin{proof}
    Let $k$ be a positive integer. Let $m$ and $n$ be positive integers such that $n>k/2$. Then
    \begin{align*}
        \tilde{r}(P_k,P_{m+n})
        &\leq \tilde{r}(P_k,P_{m+k})+\tilde{r}(P_k,P_n)+k&\text{(Lemma \ref{lem3})}\\
        &\leq \tilde{r}(P_k,P_{m})+\tilde{r}(P_k,P_n)+5k.&\text{(Lemma \ref{lem2})}
    \end{align*}
\end{proof}

In order to prove Theorem \ref{thm2}, we show that $\tilde{r}(P_k,C_{n})$ is eventually almost subadditive.

\begin{lemma}\label{lem6}
    For any positive integers $k$, $m$ and $n$ satisfying $m+n>k^2-k$ and $n>k/2$, 
    \[\tilde{r}(P_k,C_{m+n}) \leq\tilde{r}(P_k,C_{m})+\tilde{r}(P_k,C_{n})+9k.\]
\end{lemma}
\begin{proof}
    Let $k$ be a positive integer. Let $m$ and $n$ be positive integers such that $n>k/2$ and $m+n>k^2-k$. Then
    \begin{align*}
        \tilde{r}(P_k,C_{m+n})&\leq \tilde{r}(P_k,P_{m+n+k})+2k& \text{(Lemma \ref{lem4})}\\
        &\leq \tilde{r}(P_k,P_{m+2k})+\tilde{r}(P_k,P_n)+3k&\text{(Lemma \ref{lem3})}\\
        &\leq \tilde{r}(P_k,P_{m})+\tilde{r}(P_k,P_n)+9k&\text{(Lemma \ref{lem2})}\\
        &\leq \tilde{r}(P_k,C_{m})+\tilde{r}(P_k,C_n)+9k.&\text{(Lemma \ref{lem1})}\\
    \end{align*}
\end{proof}

Lemma \ref{lem5} shows that $\tilde{r}(P_k,P_n)$ is eventually almost subadditive, and Lemma \ref{lem6} shows that $\tilde{r}(P_k,C_n)$ is eventually almost subadditive. By Corollary \ref{lemfek}, $\tilde{r}(P_k,P_n)/n$ and $\tilde{r}(P_k,C_n)/n$ converge as $n$ tends to $\infty$. Lastly, we show that they converge to the same value.  

\begin{proof}[Proof of Theorem \ref{thm3}]
    Let $k$ and $n$ be positive integers such that $n>k^2$. Then
    \begin{align*}
        \tilde{r}(P_k,P_{n})&\leq \tilde{r}(P_k,C_{n})& \text{(Lemma \ref{lem1})}\\
        &\leq \tilde{r}(P_k,P_{n+k})+2k&\text{(Lemma \ref{lem4})}\\
        &\leq \tilde{r}(P_k,P_{n})+6k.&\text{(Lemma \ref{lem2})}
    \end{align*}   
    Dividing through by $n$, and applying the squeeze theorem, concludes the proof.
\end{proof}


In order to prove Theorem \ref{thm4}, we show that $\tilde{r}(K_{1,k},P_{n})$ is almost subadditive.

\begin{lemma}\label{lem4.3}
    For any positive integers $k$, $m$ and $n$, 
    \[\tilde{r}(K_{1,k},P_{m+n}) \leq\tilde{r}(K_{1,k},P_{m})+\tilde{r}(K_{1,k},P_{n})+k+k^2.\]
\end{lemma}
\begin{proof}
    Let $k$, $m$ and $n$ be positive integers. Then
    \begin{align*}
        \tilde{r}(K_{1,k},P_{m+n})
        &\leq \tilde{r}(K_{1,k},P_{m+k})+\tilde{r}(K_{1,k},P_n)+k&\text{(Lemma \ref{lem2.4})}\\
        &\leq \tilde{r}(K_{1,k},P_{m})+\tilde{r}(K_{1,k},P_n)+k+k^2&\text{(Lemma \ref{lem2.2})}
    \end{align*}
\end{proof}

In order to prove Theorem \ref{thm5}, we show that $\tilde{r}(K_{1,k},C_{n})$ is eventually almost subadditive.

\begin{lemma}\label{lem4.4}
    For any positive integers $k$, $m$ and $n$ satisfying $m+n\ge k+2$, 
    \[\tilde{r}(K_{1,k},C_{m+n}) \leq\tilde{r}(K_{1,k},C_{m})+\tilde{r}(K_{1,k},C_{n})+4k+3k^2.\]
\end{lemma}
\begin{proof}
    Let $k$ be a positive integer. Let $m$ and $n$ be positive integers such $m+n\geq k+2$. Then
    \begin{align*}
        \tilde{r}(K_{1,k},C_{m+n})&\leq \tilde{r}(K_{1,k},P_{m+n+2k})+3k& \text{(Lemma \ref{lem2.5})}\\
        &\leq \tilde{r}(K_{1,k},P_{m+3k})+\tilde{r}(K_{1,k},P_n)+4k&\text{(Lemma \ref{lem2.4})}\\
        &\leq \tilde{r}(K_{1,k},P_{m})+\tilde{r}(K_{1,k},P_n)+4k+3k^2&\text{(Lemma \ref{lem2.2})}\\
        &\leq \tilde{r}(K_{1,k},C_{m})+\tilde{r}(K_{1,k},C_n)+4k+3k^2.&\text{(Lemma \ref{lem1})}\\
    \end{align*}
\end{proof}

Lemma \ref{lem4.3} shows that $\tilde{r}(K_{1,k},P_n)$ is almost subadditive, and Lemma \ref{lem4.4} shows that $\tilde{r}(K_{1,k},C_n)$ is eventually almost subadditive. 
By Corollary \ref{lemfek}, $\tilde{r}(K_{1,k},P_n)/n$ and $\tilde{r}(K_{1,k},C_n)/n$ converge as $n$ tends to $\infty$. Lastly, we show that they converge to the same value. 

\begin{proof}[Proof of Theorem \ref{thm6}]
    Let $k$ and $n$ be positive integers such that $n\geq k+2$. Then
    \begin{align*}
        \tilde{r}(K_{1,k},P_{n})&\leq \tilde{r}(K_{1,k},C_{n})& \text{(Lemma \ref{lem1})}\\
        &\leq \tilde{r}(K_{1,k},P_{n+2k})+3k&\text{(Lemma \ref{lem2.5})}\\
        &\leq \tilde{r}(K_{1,k},P_{n})+3k+2k^2.&\text{(Lemma \ref{lem2.2})}
    \end{align*}   
    Dividing everything by $n$ and applying the squeeze theorem, concludes the proof.
\end{proof}

%% file: remarks.tex
\section{Future Work}

We not only show that the asymptotic values of $\tilde{r}(P_k,P_n)$ and $\tilde{r}(P_k,C_n)$ exist, but that they are also equal. Our approach for proving existence revolves around proving almost subadditivity and applying a stronger version of Fekete's Lemma. On the other hand, equivalence comes from showing that, in a relatively small number of rounds, a blue copy of $C_n$ or a red copy of $P_k$ can be forced starting with a blue copy of $P_n$. We use this same approach to show that the asymptotic values of $\tilde{r}(K_{1,k},P_n)$ and $\tilde{r}(K_{1,k},C_n)$ exist and are equal.

In terms of showing that the asymptotic values exist, the only unresolved off-diagonal case for paths and cycles is for the online Ramsey number $\tilde{r}(C_k,P_n)$ when $k$ is odd. Unfortunately, it can be shown via a Painter strategy that our approach does not work for this case. Our approach also does not work for the diagonal cases $\tilde{r}(P_n,P_n)$ and $\tilde{r}(C_n,C_n)$. However, we believe that each of these asymptotic values exist, and finding a new approach to show it would be interesting.

We leave the reader with one last problem to consider. Assuming that both limits exist, for which graphs $H$ do the online Ramsey numbers $\tilde{r}(H,P_n)$ and $\tilde{r}(H,C_n)$ have the same asymptotic value? 

%% file: ramsey.bib
@article {ada-bed-24,
    AUTHOR = {Adamski, Grzegorz and Bednarska-Bzd\c{e}ga, Ma\l{}gorzata},
     TITLE = {Online size {R}amsey numbers: odd cycles vs connected graphs},
   JOURNAL = {Electron. J. Combin.},
  FJOURNAL = {Electronic Journal of Combinatorics},
    VOLUME = {31},
      YEAR = {2024},
    NUMBER = {3},
     PAGES = {Paper No. 3.16, 15},
      ISSN = {1077-8926},
   MRCLASS = {05C57 (05C55 91A46)},
  MRNUMBER = {4790336},
       DOI = {10.37236/11644},
       URL = {https://doi-org.kean.idm.oclc.org/10.37236/11644},
}

@article {ada-bed-bla-24,
    AUTHOR = {Adamski, Grzegorz and Bednarska-Bzd\c{e}ga, Ma\l{}gorzata and
              Bla\v{z}ej, V\'aclav},
     TITLE = {Online {R}amsey numbers: long versus short cycles},
   JOURNAL = {SIAM J. Discrete Math.},
  FJOURNAL = {SIAM Journal on Discrete Mathematics},
    VOLUME = {38},
      YEAR = {2024},
    NUMBER = {4},
     PAGES = {3150--3175},
      ISSN = {0895-4801,1095-7146},
   MRCLASS = {05C55 (05C38 91A46)},
  MRNUMBER = {4839679},
MRREVIEWER = {Stanis\l aw\ P.\ Radziszowski},
       DOI = {10.1137/23M156183X},
       URL = {https://doi-org.kean.idm.oclc.org/10.1137/23M156183X},
}

@article {bec-83,
    AUTHOR = {Beck, J\'ozsef},
     TITLE = {On size {R}amsey number of paths, trees, and circuits. {I}},
   JOURNAL = {J. Graph Theory},
  FJOURNAL = {Journal of Graph Theory},
    VOLUME = {7},
      YEAR = {1983},
    NUMBER = {1},
     PAGES = {115--129},
      ISSN = {0364-9024,1097-0118},
   MRCLASS = {05C55},
  MRNUMBER = {693028},
MRREVIEWER = {Saul\ Stahl},
       DOI = {10.1002/jgt.3190070115},
       URL = {https://doi-org.kean.idm.oclc.org/10.1002/jgt.3190070115},
}

@article {bed-24,
    AUTHOR = {Bednarska-Bzd\c{e}ga, Ma\l{}gorzata},
     TITLE = {Off-diagonal online size {R}amsey numbers for paths},
   JOURNAL = {European J. Combin.},
  FJOURNAL = {European Journal of Combinatorics},
    VOLUME = {118},
      YEAR = {2024},
     PAGES = {Paper No. 103873, 16},
      ISSN = {0195-6698,1095-9971},
   MRCLASS = {05C55 (05C38)},
  MRNUMBER = {4668780},
MRREVIEWER = {Stanis\l aw\ P.\ Radziszowski},
       DOI = {10.1016/j.ejc.2023.103873},
       URL = {https://doi-org.kean.idm.oclc.org/10.1016/j.ejc.2023.103873},
}

@incollection {bla-dvo-val-19,
    AUTHOR = {Bla\v{z}ej, V\'aclav and Dvo\v{r}\'ak, Pavel and Valla, Tom\'a\v{s}},
     TITLE = {On induced online {R}amsey number of paths, cycles, and trees},
 BOOKTITLE = {Computer science---theory and applications},
    SERIES = {Lecture Notes in Comput. Sci.},
    VOLUME = {11532},
     PAGES = {60--69},
 PUBLISHER = {Springer, Cham},
      YEAR = {2019},
      ISBN = {978-3-030-19955-5; 978-3-030-19954-8},
   MRCLASS = {05C55},
  MRNUMBER = {3976039},
       DOI = {10.1007/978-3-030-19955-5\_6},
       URL = {https://doi-org.kean.idm.oclc.org/10.1007/978-3-030-19955-5_6},
}

@article{bru-erd-52,
  author  = {De Bruijn, N. G. and Erd\H{o}s, P.},
  title   = {Some linear and some quadratic recursion formulas. I},
  journal = {Proc. Koninklijke Nederlandse Akademie van Wetenschappen, Series A},
  volume  = {55},
  year    = {1952},
  pages   = {374--382}
}

@article {cdll-15,
    AUTHOR = {Cyman, Joanna and Dzido, Tomasz and Lapinskas, John and Lo,
              Allan},
     TITLE = {On-line {R}amsey numbers of paths and cycles},
   JOURNAL = {Electron. J. Combin.},
  FJOURNAL = {Electronic Journal of Combinatorics},
    VOLUME = {22},
      YEAR = {2015},
    NUMBER = {1},
     PAGES = {Paper 1.15, 32},
      ISSN = {1077-8926},
   MRCLASS = {05C55 (05C57 91A43 91A46)},
  MRNUMBER = {3315457},
MRREVIEWER = {Vera\ Rosta},
       DOI = {10.37236/4097},
       URL = {https://doi-org.kean.idm.oclc.org/10.37236/4097},
}

@article {efrs-78,
    AUTHOR = {Erd\H{o}s, P. and Faudree, R. J. and Rousseau, C. C. and
              Schelp, R. H.},
     TITLE = {The size {R}amsey number},
   JOURNAL = {Period. Math. Hungar.},
  FJOURNAL = {Periodica Mathematica Hungarica. Journal of the J\'anos Bolyai
              Mathematical Society},
    VOLUME = {9},
      YEAR = {1978},
    NUMBER = {1-2},
     PAGES = {145--161},
      ISSN = {0031-5303,1588-2829},
   MRCLASS = {05C55},
  MRNUMBER = {479691},
MRREVIEWER = {F.\ Harary},
       DOI = {10.1007/BF02018930},
       URL = {https://doi-org.kean.idm.oclc.org/10.1007/BF02018930},
}

@article{fek-23,
  author  = {Michael Fekete},
  title   = {Über die Verteilung der Wurzeln bei gewissen algebraischen Gleichungen mit ganzzahligen Koeffizienten},
  journal = {Mathematische Zeitschrift},
  volume  = {17},
  year    = {1923},
  pages   = {228--249}
}

@article {gry-kie-pra-08,
    AUTHOR = {Grytczuk, J. A. and Kierstead, H. A. and Pra\l{}at, P.},
     TITLE = {On-line {R}amsey numbers for paths and stars},
   JOURNAL = {Discrete Math. Theor. Comput. Sci.},
  FJOURNAL = {Discrete Mathematics \& Theoretical Computer Science. DMTCS.},
    VOLUME = {10},
      YEAR = {2008},
    NUMBER = {3},
     PAGES = {63--74},
      ISSN = {1365-8050},
   MRCLASS = {05C55 (05C15 68R10 91A43)},
  MRNUMBER = {2445473},
MRREVIEWER = {Stanis\l aw\ P.\ Radziszowski},
}

@article {lat-tan-21,
    AUTHOR = {Mohd Latip, Fatin Nur Nadia Binti and Tan, Ta Sheng},
     TITLE = {A note on on-line {R}amsey numbers of stars and paths},
   JOURNAL = {Bull. Malays. Math. Sci. Soc.},
  FJOURNAL = {Bulletin of the Malaysian Mathematical Sciences Society},
    VOLUME = {44},
      YEAR = {2021},
    NUMBER = {5},
     PAGES = {3511--3521},
      ISSN = {0126-6705,2180-4206},
   MRCLASS = {05C55 (05C57 05D10)},
  MRNUMBER = {4296347},
MRREVIEWER = {Alexander\ Daniel\ Halperin},
       DOI = {10.1007/s40840-021-01130-x},
       URL = {https://doi-org.kean.idm.oclc.org/10.1007/s40840-021-01130-x},
}

@article {pra-08,
    AUTHOR = {Pra\l{}at, Pawe\l},
     TITLE = {A note on small on-line {R}amsey numbers for paths and their
              generalization},
   JOURNAL = {Australas. J. Combin.},
  FJOURNAL = {The Australasian Journal of Combinatorics},
    VOLUME = {40},
      YEAR = {2008},
     PAGES = {27--36},
      ISSN = {1034-4942,2202-3518},
   MRCLASS = {05C55 (05-04)},
  MRNUMBER = {2381412},
MRREVIEWER = {Andr\'as\ Gy\'arf\'as},
}

@article {pra-12,
    AUTHOR = {Pra\l{}at, Pawe\l},
     TITLE = {A note on off-diagonal small on-line {R}amsey numbers for
              paths},
   JOURNAL = {Ars Combin.},
  FJOURNAL = {Ars Combinatoria. A Canadian Journal of Combinatorics},
    VOLUME = {107},
      YEAR = {2012},
     PAGES = {295--306},
      ISSN = {0381-7032,2817-5204},
   MRCLASS = {05C55},
  MRNUMBER = {3013008},
}

@article {mon-por-24,
    AUTHOR = {Mond, Adva and Portier, Julien},
     TITLE = {The asymptotic of off-diagonal online {R}amsey numbers for
              paths},
   JOURNAL = {European J. Combin.},
  FJOURNAL = {European Journal of Combinatorics},
    VOLUME = {122},
      YEAR = {2024},
     PAGES = {Paper No. 104032, 12},
      ISSN = {0195-6698,1095-9971},
   MRCLASS = {05C55 (05C38 05C57)},
  MRNUMBER = {4783966},
MRREVIEWER = {Tomasz\ Dzido},
       DOI = {10.1016/j.ejc.2024.104032},
       URL = {https://doi-org.kean.idm.oclc.org/10.1016/j.ejc.2024.104032},
}

@article {son-wan-zha-25,
    AUTHOR = {Song, Ruyu and Wang, Sha and Zhang, Yanbo},
     TITLE = {Online {R}amsey numbers of {$K_{1, 3}$} versus paths},
   JOURNAL = {Discrete Appl. Math.},
  FJOURNAL = {Discrete Applied Mathematics. The Journal of Combinatorial
              Algorithms, Informatics and Computational Sciences},
    VOLUME = {377},
      YEAR = {2025},
     PAGES = {218--224},
      ISSN = {0166-218X,1872-6771},
   MRCLASS = {05C55},
  MRNUMBER = {4929521},
MRREVIEWER = {Tomasz\ Dzido},
       DOI = {10.1016/j.dam.2025.06.050},
       URL = {https://doi-org.kean.idm.oclc.org/10.1016/j.dam.2025.06.050},
}

@misc{zha-zha-23,
      title={Proof of a Conjecture on Online Ramsey Numbers of Paths}, 
      author={Yanbo Zhang and Yixin Zhang},
      year={2023},
      eprint={2302.13640},
      archivePrefix={arXiv},
      primaryClass={math.CO},
      url={https://arxiv.org/abs/2302.13640}, 
}

@misc{zhi-zha-26,
      title={Online Ramsey numbers of the claw versus cycles}, 
      author={Hexuan Zhi and Yanbo Zhang},
      year={2026},
      eprint={2601.05452},
      archivePrefix={arXiv},
      primaryClass={math.CO},
      url={https://arxiv.org/abs/2601.05452}, 
}
